\documentclass[reqno]{amsart}

\usepackage[usenames,dvipsnames]{pstricks}
\usepackage{epsfig}

\usepackage{color}
\date{\today}

\usepackage[latin1]{inputenc}
\usepackage[T1]{fontenc}
\usepackage{amsmath, amsthm}
\usepackage[english]{babel}
\usepackage{amssymb}
\usepackage{latexsym}
\usepackage{graphicx}
\usepackage[all]{xy}

\newtheorem{theorem}{Theorem}[section]
\newtheorem{lemma}[theorem]{Lemma}
\newtheorem{proposition}[theorem]{Proposition}
\newtheorem{definition}[theorem]{Definition}

\newtheorem{corollary}[theorem]{Corollary}

\newtheorem{remark}[theorem]{Remark}

\newcommand{\ot}{\otimes}
\newcommand{\co}{\circ}

\begin{document}

\begin{center}

{\huge{\bf A characterization of weak Hopf (co)quasigroups}}

\end{center}

\ \\
\begin{center}
{\bf J.N. Alonso \'Alvarez$^{1}$, J.M. Fern\'andez Vilaboa$^{2}$, R.
Gonz\'{a}lez Rodr\'{\i}guez$^{3}$}
\end{center}

\ \\
\hspace{-0,5cm}$^{1}$ Departamento de Matem\'{a}ticas, Universidad
de Vigo, Campus Universitario Lagoas-Marcosende, E-36280 Vigo, Spain
(e-mail: jnalonso@ uvigo.es)
\ \\
\hspace{-0,5cm}$^{2}$ Departamento de \'Alxebra, Universidad de
Santiago de Compostela.  E-15771 Santiago de Compostela, Spain
(e-mail: josemanuel.fernandez@usc.es)
\ \\
\hspace{-0,5cm}$^{3}$ Departamento de Matem\'{a}tica Aplicada II,
Universidad de Vigo, Campus Universitario Lagoas-Marcosende, E-36310
Vigo, Spain (e-mail: rgon@dma.uvigo.es)
\ \\

{\bf Abstract} In this paper we show that weak Hopf (co)quasigroups can be characterized by a Galois-type condition. Taking into account that this notion generalizes the ones of Hopf (co)quasigroup and weak Hopf algebra, we obtain as a consequence the first fundamental theorem for Hopf (co)quasigroups and a characterization of weak Hopf algebras in terms of bijectivity of a Galois-type morphism (also called fusion morphism).

\vspace{0.5cm}



\section{introduction}

The notion of Hopf algebra and its generalizations appeared  as useful tools 
in relation with many branch of mathematics such that algebraic geometry, number theory, 
Lie theory, Galois theory, quantum group theory and so on. A common principle to obtain generalizations of the original notion of Hopf algebra is to weak some of  axioms of its definition.  For example, if one does not force the coalgebra structure to respect the unit of the algebra structure, one is lead to weak Hopf algebras. In a different way, the weakening of the associativity leads to Hopf quasigroups and quasi-Hopf  algebras. 

Weak Hopf algebras (or quantum groupoids in the
terminology of Nikshych and Vainerman \cite{NV}) were introduced
by B\"{o}hm, Nill and Szlach\'anyi \cite{bohm} as a new
generalization of Hopf algebras and groupoid algebras. A weak Hopf
algebra $H$ in a braided monoidal category \cite{IND} is an object that has
both, monoid and comonoid structure, with some relations between
them. The main difference with other Hopf algebraic constructions is that weak
Hopf algebras are coassociative but the coproduct is not required
to preserve the unit, equivalently, the counit is
not a monoid morphism.
 Some motivations to study weak Hopf
algebras come from the following facts: firstly, as group algebras
and their duals are the natural examples of Hopf algebras,
groupoid algebras and their duals provide examples of weak Hopf
algebras and, secondly, these algebraic structures have a
remarkable connection with the theory of algebra extensions,
important applications in the study of dynamical twists of Hopf
algebras and a deep link with quantum field theories and operator
algebras \cite{NV}.

On the other hand, Hopf (co)quasigroups were introduced in \cite{Majidesfera} in order to
understand the structure and relevant properties of the algebraic
$7$-sphere. They are a non-(co)associative generalizations of Hopf algebras.
 Like in the quasi-Hopf setting, Hopf quasigroups are not
associative but the lack of this property is compensated by some
axioms involving the antipode. The concept of Hopf quasigroup is a
particular instance of the notion of unital coassociative
$H$-bialgebra introduced in \cite{PI2}.

Recently \cite{AFG-Weak-quasi}, the authors have introduced a new generalization of Hopf algebras (called weak Hopf (co)quasigroups) which encompass weak Hopf algebras and Hopf (co)quasigroups. A family of non-trivial examples of weak Hopf quasigroups can be obtained by working with bigroupoids, i.e. bicategories where every $1$-cell is an equivalence and every $2$-cell is an isomorphism. Moreover, many properties of these algebraic structures remain valid under this unified approach (in particular, the Fundamental Theorem of Hopf Modules associated to a weak  Hopf quasigroup \cite{AFG-Weak-quasi}), and it is very natural to ask for other well-known properties related with Hopf algebras. In particular, Nakajima \cite{N} gave a characterization of ordinary Hopf algebras in terms of bijectivity of right or left Galois maps (also called fusion morphisms in \cite{Street}). This result was extended by Schauenburg \cite{S} to weak Hopf algebras, and by Brzezi\'nski \cite{Brz} to Hopf (co)quasigroups. The main purpose of this work is to give a similar characterization in the weak Hopf (co)quasigroup setting. More precisely, we state that a weak Hopf (co)quasigroup satisfies a right and left Galois-type condition, and these Galois morphisms must have almost right and left (co)linear inverses, and conversely. As a consequence we get the characterization of Hopf (co)quasigroups given by Brzezi\'nski \cite{Brz} (called the First Fundamental Theorem for Hopf (co)quasigroups), and the one obtained by Schauenburg \cite{S} for weak Hopf algebras.

\section{A characterization of weak Hopf quasigroups}

Throughout this paper $\mathcal C$ denotes a
strict  monoidal category with tensor product $\ot$
and unit object $K$. For each object $M$ in  $ \mathcal
C$, we denote the identity morphism by $id_{M}:M\rightarrow M$ and, for
simplicity of notation, given objects $M$, $N$, $P$ in $\mathcal
C$ and a morphism $f:M\rightarrow N$, we write $P\ot f$ for
$id_{P}\ot f$ and $f \ot P$ for $f\ot id_{P}$.

From now on we also assume that $\mathcal C$  admits equalizers and coequalizers. Then every idempotent morphism splits, i.e., for every morphism $\nabla_{Y}:Y\rightarrow Y$ such that $\nabla_{Y}=\nabla_{Y}\co\nabla_{Y}$, there exist an object $Z$ and morphisms $i_{Y}:Z\rightarrow Y$ and $p_{Y}:Y\rightarrow Z$ such that $\nabla_{Y}=i_{Y}\co p_{Y}$ and $p_{Y}\co i_{Y} =id_{Z}$. 

Also we assume that $\mathcal C$ is braided, that is:  for all $M$ and $N$ objects in $\mathcal C$,
there is a natural isomorphism $c_{M, N}:M\ot N\rightarrow N\ot M$,
called the braiding, satisfying the Hexagon Axiom (see \cite{JS}
for generalities). If the braiding satisfies $c_{N,M}\co
c_{M,N}=id_{M\ot N}$, the category $\mathcal C$ will be called
symmetric.

By a unital  magma in $\mathcal C$ we understand a triple $A=(A, \eta_{A}, \mu_{A})$ where $A$ is an object in $\mathcal C$ and $\eta_{A}:K\rightarrow A$ (unit), $\mu_{A}:A\ot A \rightarrow A$ (product) are morphisms in $\mathcal C$ such that $\mu_{A}\co (A\ot \eta_{A})=id_{A}=\mu_{A}\co (\eta_{A}\ot A)$. If $\mu_{A}$ is associative, that is, $\mu_{A}\co (A\ot \mu_{A})=\mu_{A}\co (\mu_{A}\ot A)$, the unital magma will be called a monoid in $\mathcal C$.   Given two unital magmas
(monoids) $A= (A, \eta_{A}, \mu_{A})$ and $B=(B, \eta_{B}, \mu_{B})$, $f:A\rightarrow B$ is a morphism of unital magmas (monoids)  if $\mu_{B}\co (f\ot f)=f\co \mu_{A}$ and $ f\co \eta_{A}= \eta_{B}$. 

By duality, a counital comagma in $\mathcal C$ is a triple ${D} = (D, \varepsilon_{D}, \delta_{D})$ where $D$ is an object in $\mathcal C$ and $\varepsilon_{D}: D\rightarrow K$ (counit), $\delta_{D}:D\rightarrow D\ot D$ (coproduct) are morphisms in $\mathcal C$ such that $(\varepsilon_{D}\ot D)\co \delta_{D}= id_{D}=(D\ot \varepsilon_{D})\co \delta_{D}$. If $\delta_{D}$ is coassociative, that is, $(\delta_{D}\ot D)\co \delta_{D}= (D\ot \delta_{D})\co \delta_{D}$, the counital comagma will be called a comonoid. If ${D} = (D, \varepsilon_{D}, \delta_{D})$ and ${ E} = (E, \varepsilon_{E}, \delta_{E})$ are counital comagmas
(comonoids), $f:D\rightarrow E$ is  a morphism of counital comagmas (comonoids) if $(f\ot f)\co \delta_{D} =\delta_{E}\co f$ and  $\varepsilon_{E}\co f =\varepsilon_D.$

If  $A$, $B$ are unital magmas (monoids) in $\mathcal C$, the object $A\ot B$ is a unital  magma (monoid) in $\mathcal C$ where $\eta_{A\ot B}=\eta_{A}\ot \eta_{B}$ and $\mu_{A\ot B}=(\mu_{A}\ot \mu_{B})\co (A\ot c_{B,A}\ot B).$  In a dual way, if $D$, $E$ are counital comagmas (comonoids) in $\mathcal C$, $D\ot E$ is a  counital comagma (comonoid) in $\mathcal C$ where $\varepsilon_{D\ot E}=\varepsilon_{D}\ot \varepsilon_{E}$ and $\delta_{D\ot
E}=(D\ot c_{D,E}\ot E)\co( \delta_{D}\ot \delta_{E}).$

Moreover, if $D$ is a comagma and $A$ a magma, given two morphisms $f,g:D\rightarrow A$ we will denote by $f\ast g$ its convolution product in $\mathcal C$, that is 
$f\ast g=\mu_{A}\co (f\ot g)\co \delta_{D}$.

Let $A$ be a monoid. The pair $(M, \phi_M)$ is a right $A$-module if $M$ is an object in $\mathcal C$ and $\phi_M:A\ot M\rightarrow M$ is a morphism in $\mathcal C$ such that $\phi_M\co (\eta_A\ot M)=id_M$ and $\phi_M\co (A\ot \phi_M)=\phi_M\co (\mu_A\ot M)$. Given two right $A$-modules $(M, \phi_M)$ and $(N, \phi_N$, a map $f:M\rightarrow N$ is a morphism of right $A$-modules if $\phi_N\co (A\ot f)=f\co \phi_M$. We shall denote by $\mathcal {C}_A$ the category of right $A$-modules. In an analogous way we can define the category of left $A$-modules and we denote it by $_A\mathcal {C}$.

 Let $D$ be a comonoid. The pair $(M, \rho_M)$ is a right $D$-comodule if $M$ is an object in $\mathcal C$ and $\rho_M:M\rightarrow M\ot D$ is a morphism in $\mathcal C$ satisfying that $(M\ot \varepsilon_D)\co \rho_M=id_M$ and $(\rho_M\ot D)\co \rho_M=(M\ot \delta_D)\co \rho_M$. Given two right $D$-comodules $(M, \rho_M)$ and $(N, \rho_N)$, a map $f:M\rightarrow N$ is a morphism of right $D$-comodules if $(f\ot D)\co \rho_M=\rho_N\co g$. We shall denote by $\mathcal {C}^D$ the category of right $D$-comodules. In an analogous way we can define the category of left $D$-comodules and we denote it by $^D\mathcal {C}$.

Now we recall the notion of weak Hopf quasigroup we introduced in \cite{AFG-Weak-quasi}.

\begin{definition}
\label{Weak-Hopf-quasigroup} {\rm  A weak Hopf quasigroup $H$   in ${\mathcal
C}$ is a unital magma $(H, \eta_H, \mu_H)$ and a comonoid $(H,\varepsilon_H, \delta_H)$ such that the following axioms hold:
\begin{itemize}
\item[(a1)] $\delta_{H}\co \mu_{H}=(\mu_{H}\ot \mu_{H})\co \delta_{H\ot H}.$
\item[(a2)] $\varepsilon_{H}\co \mu_{H}\co (\mu_{H}\ot
H)=\varepsilon_{H}\co \mu_{H}\co (H\ot \mu_{H})$
\item[ ]$= ((\varepsilon_{H}\co \mu_{H})\ot (\varepsilon_{H}\co
\mu_{H}))\co (H\ot \delta_{H}\ot H)$ 
\item[ ]$=((\varepsilon_{H}\co \mu_{H})\ot (\varepsilon_{H}\co
\mu_{H}))\co (H\ot (c_{H,H}^{-1}\co\delta_{H})\ot H).$
\item[(a3)]$(\delta_{H}\ot H)\co \delta_{H}\co \eta_{H}=(H\ot
\mu_{H}\ot H)\co ((\delta_{H}\co \eta_{H}) \ot (\delta_{H}\co
\eta_{H}))$  \item[ ]$=(H\ot (\mu_{H}\co c_{H,H}^{-1})\ot
H)\co ((\delta_{H}\co \eta_{H}) \ot (\delta_{H}\co \eta_{H})).$

\item[(a4)] There exists a morphism $\lambda_{H}:H\rightarrow H$ (called the antipode of $H$) such that, if we denote by $\Pi_{H}^{L}$ (target morphism) and by $\Pi_{H}^{R}$ (source morphism) the morphisms 
$$\Pi_{H}^{L}=((\varepsilon_{H}\co\mu_{H})\ot H)\co (H\ot c_{H,H})\co ((\delta_{H}\co \eta_{H})\ot H),$$
$$\Pi_{H}^{R}=(H\ot(\varepsilon_{H}\co \mu_{H}))\co (c_{H,H}\ot H)\co (H\ot (\delta_{H}\co \eta_{H})),$$
then:
\begin{itemize}
\item[(a4-1)] $\Pi_{H}^{L}=id_{H}\ast \lambda_{H}.$
\item[(a4-2)] $\Pi_{H}^{R}=\lambda_{H}\ast id_{H}.$
\item[(a4-3)] $\lambda_{H}\ast \Pi_{H}^{L}=\Pi_{H}^{R}\ast \lambda_{H}= \lambda_{H}.$
\item[(a4-4)] $\mu_H\circ (\lambda_H\ot \mu_H)\circ (\delta_H\ot H)=\mu_{H}\co (\Pi_{H}^{R}\ot H).$
\item[(a4-5)] $\mu_H\circ (H\ot \mu_H)\circ (H\ot \lambda_H\ot H)\circ (\delta_H\ot H)=\mu_{H}\co (\Pi_{H}^{L}\ot H).$
\item[(a4-6)] $\mu_H\circ(\mu_H\ot \lambda_H)\circ (H\ot \delta_H)=\mu_{H}\co (H\ot \Pi_{H}^{L}).$
\item[(a4-7)] $\mu_H\circ (\mu_H\ot H)\circ (H\ot \lambda_H\ot H)\circ (H\ot \delta_H)=\mu_{H}\co (H\ot \Pi_{H}^{R}).$
\end{itemize}

\end{itemize}

Note that, if in the previous definition the triple $(H, \eta_H, \mu_H)$ is a monoid, we obtain the notion of weak Hopf algebra in a braided category introduced in \cite{AFG1} (see also \cite{IND}). Under this assumption, if ${\mathcal C}$ is symmetric, we have the monoidal version of the original definition of weak Hopf algebra introduced by B\"{o}hm, Nill and Szlach\'anyi in \cite{bohm}. On the other hand, if  $\varepsilon_H$ and $\delta_H$ are  morphisms of unital magmas, (equivalently, $\eta_{H}$, $\mu_{H}$ are morphisms of counital comagmas), $\Pi_{H}^{L}=\Pi_{H}^{R}=\eta_{H}\ot \varepsilon_{H}$ and, as a consequence, we have  the notion of Hopf quasigroup defined  by Klim and Majid in \cite{Majidesfera} in the category of vector spaces over a field ${\Bbb F}$. (Note that in this case there is no difference between the definitions for the symmetric and the braided  settings).
}
\end{definition}

Now we recall some properties related with weak Hopf quasigroups we will need in what sequel. The proofs are identical to the ones given in \cite{AFG-Weak-quasi}, because condition (a4) of Definition \ref{Weak-Hopf-quasigroup} is unnecessary.

\begin{proposition}
\label{propertieswithouta4}
Let $H$ be a unital magma and comonoid such that conditions (a1), (a2) and (a3) of Definition \ref{Weak-Hopf-quasigroup} hold. Define $\overline{\Pi}_{H}^{L}$ and $\overline{\Pi}_{H}^{R}$ by 
$$\overline{\Pi}_{H}^{L}=(H\ot (\varepsilon_{H}\co \mu_{H}))\co ((\delta_{H}\co \eta_{H})\ot H)$$
and 
$$\overline{\Pi}_{H}^{R}=((\varepsilon_{H}\co \mu_{H})\ot H)\co (H\ot (\delta_{H}\co \eta_{H})).$$
Then the morphisms $\Pi_{H}^{L}$, $\Pi_{H}^{R}$, $\overline{\Pi}_{H}^{L}$ and 
$\overline{\Pi}_{H}^{R}$ are idempotent. Moreover, the following equalities
\begin{equation}
\label{unidadpi}
\Pi_{H}^{L}\co \eta_H=\Pi_{H}^{R}\co \eta_H=\overline{\Pi}_{H}^{L}\co \eta_H=\overline{\Pi}_{H}^{R}\co \eta_H=\eta_H,
\end{equation}
\begin{equation}
\label{counidadpi}
\varepsilon_H\co \Pi_{H}^{L}=\varepsilon_H\co \Pi_{H}^{R}=\varepsilon_H\co \overline{\Pi}_{H}^{L}=\varepsilon_H\co \overline{\Pi}_{H}^{R}=\varepsilon_H,
\end{equation}
\begin{equation}
\label{pi-l}
\Pi_{H}^{L}\ast id_{H}=id_{H}\ast \Pi_{H}^{R}=id_{H},
\end{equation}
\begin{equation}
\label{mu-pi-l}
\mu_{H}\co (H\ot \Pi_{H}^{L})=((\varepsilon_{H}\co
\mu_{H})\ot H)\co (H\ot c_{H,H})\co (\delta_{H}\ot H),
\end{equation}
\begin{equation}
\label{mu-pi-r}
\mu_{H}\co (\Pi_{H}^{R}\ot H)=(H\ot(\varepsilon_{H}\co \mu_{H}))\co (c_{H,H}\ot H)\co
(H\ot \delta_{H}),
\end{equation}
\begin{equation}
\label{mu-pi-l-var}
\mu_{H}\co (H\ot \overline{\Pi}_{H}^{L})=(H\ot (\varepsilon_{H}\co
\mu_{H}))\co (\delta_{H}\ot H),
\end{equation}
\begin{equation}
\label{mu-pi-r-var}
\mu_{H}\co (\overline{\Pi}_{H}^{R}\ot H)=((\varepsilon_{H}\co
\mu_{H})\ot H)\co (H\ot \delta_{H}), 
\end{equation}
\begin{equation}
\label{delta-pi-l}
 (H\ot \Pi_{H}^{L})\co \delta_{H}=(\mu_{H}\ot H)\co (H\ot c_{H,H})\co ((\delta_{H}\co \eta_{H})\ot H),
\end{equation}
\begin{equation}
\label{delta-pi-r}
(\Pi_{H}^{R}\ot H)\co \delta_{H}=(H\ot \mu_{H})\co (c_{H,H}\ot H)\co (H\ot (\delta_{H}\co \eta_{H})),
\end{equation}
\begin{equation}
\label{pi-l-mu-pi-l}
 \Pi_{H}^{L}\co \mu_{H}\co (H\ot \Pi_{H}^{L})=\Pi_{H}^{L}\co \mu_{H}=\Pi_{H}^{L}\co \mu_{H}\co (H\ot  \overline{\Pi}_{H}^{L}),
\end{equation}
\begin{equation}
\label{pi-delta-mu-pi-3}
(H\ot \Pi^{L}_{H})\circ \delta_{H}\circ\Pi^{L}_{H}=\delta_{H}\circ \Pi^{L}_{H}=(H\ot \overline{\Pi}^{R}_{H})\circ \delta_{H}\circ\Pi^{L}_{H},
\end{equation}
\begin{equation}
\label{pi-l-barra-delta}
 (\overline{\Pi}_{H}^{L}\ot H)\co \delta_{H}=(H\ot \mu_{H})\co ((\delta_{H}\co \eta_{H})\ot H),
\end{equation}
\begin{equation}
\label{pi-r-barra-delta}
 (H\ot \overline{\Pi}_{H}^{R})\co \delta_{H}=(\mu_{H}\ot H)\co (H\ot (\delta_{H}\co \eta_{H})),
\end{equation}
\begin{equation}
\label{pi-delta-mu-pi-4}
(\Pi^{R}_{H}\ot H)\circ \delta_{H}\circ \Pi^{R}_{H}=\delta_{H}\circ \Pi^{R}_{H}=(\overline{\Pi}^{L}_{H}\ot H)\co \delta_H\co \Pi^{R}_{H}, 
\end{equation}
\begin{equation}
\label{doblepiLmu}
\mu_H\co (\Pi^{L}_{H}\ot \Pi^{L}_{H})=\Pi^{L}_{H}\co \mu_H\co (\Pi^{L}_{H}\ot \Pi^{L}_{H}), 
\end{equation}
\begin{equation}
\label{doblepiRmu}
\mu_H\co (\Pi^{R}_{H}\ot \Pi^{R}_{H})=\Pi^{R}_{H}\co \mu_H\co (\Pi^{R}_{H}\ot \Pi^{R}_{H}), 
\end{equation}
\begin{equation}
\label{pi-composition-2}
\overline{\Pi}_{H}^{R}\co\Pi_{H}^{L}=\Pi_{H}^{L},\;\;\; \overline{\Pi}_{H}^{L}\co \Pi_{H}^{R}=\Pi_{H}^{R},\;\;\; 
\Pi_{H}^{L}\co \overline{\Pi}_{H}^{L}=\Pi_{H}^{L}\;\;\;,\overline{\Pi}_{H}^{R}\co\Pi_{H}^{L}=\Pi_{H}^{L},
\end{equation}
hold.

\end{proposition}

The following properties are also proved in \cite{AFG-Weak-quasi}, but we give a slightly different proof without using (a4) of Definition \ref{Weak-Hopf-quasigroup}.  
\begin{proposition}
\label{otherproperties}
Let $H$ be a unital magma and comonoid such that conditions (a1), (a2) and (a3) of Definition \ref{Weak-Hopf-quasigroup} hold. Then we have that
\begin{equation}
\label{PiLRconvolution}
 \Pi_{H}^{L}\ast \Pi_{H}^{L}=\Pi_{H}^{L},\;\;\;, \Pi_{H}^{R}\ast \Pi_{H}^{R}=\Pi_{H}^{R}
\end{equation}
\begin{equation}
\label{aux-1-monoid-hl}
\delta_{H}\co \mu_{H}\co (\Pi_{H}^{L}\ot H)=(\mu_{H}\ot H)\co (\Pi_{H}^{L}\ot \delta_{H}), 
\end{equation}
\begin{equation}
\label{aux-2-monoid-hl}
\delta_{H}\co \mu_{H}\co (H\ot \Pi_{H}^{L})=(\mu_{H}\ot H)\co (H\ot c_{H,H})\co  (\delta_{H}\ot \Pi_{H}^{L}). 
\end{equation}
\begin{equation}
\label{aux-1-monoid-hr}
\delta_{H}\co \mu_{H}\co (H\ot \Pi_{H}^{R})=(H\ot \mu_{H})\co (\delta_{H}\ot \Pi_{H}^{R}), 
\end{equation}
\begin{equation}
\label{aux-2-monoid-hr}
\delta_{H}\co \mu_{H}\co (\Pi_{H}^{R}\ot H)=(H\ot \mu_{H})\co (c_{H,H}\ot H)\co  (\Pi_{H}^{R}\ot \delta_{H}). 
\end{equation}
\begin{equation}
\label{monoid-hl-1}
\mu_{H}\co ((\mu_{H}\co (\Pi_{H}^{L}\ot H))\ot H)=\mu_{H}\co (\Pi_{H}^{L}\ot \mu_{H}), 
\end{equation}
\begin{equation}
\label{monoid-hl-2}
\mu_{H}\co (H\ot (\mu_{H}\co (\Pi_{H}^{L}\ot H)))=\mu_{H}\co ((\mu_{H}\co (H\ot \Pi_{H}^{L}))\ot H), 
\end{equation}
\begin{equation}
\label{monoid-hl-3}
\mu_{H}\co (H\ot (\mu_{H}\co (H\ot \Pi_{H}^{L})))=\mu_{H}\co (\mu_{H}\ot \Pi_{H}^{L}), 
\end{equation}
and similar equalities to (\ref{monoid-hl-1}), (\ref{monoid-hl-2}) and (\ref{monoid-hl-3}) with $\Pi_{H}^{R}$ instead of $\Pi_{H}^{L}$ also hold.

\end{proposition}

\begin{proof}
 We begin by showing the first equality of (\ref{PiLRconvolution}), the second one is similar. Using the definition of $\Pi_{H}^{L}$, the naturalness of $c$, and (a2) and (a1) of Definition \ref{Weak-Hopf-quasigroup}, 

\begin{itemize}
\item[ ]$\hspace{0.38cm} \Pi_{H}^{L}\ast \Pi_{H}^{L}$

\item[ ]$=(\varepsilon_{H}\ot H)\co \mu_{H\ot H}\co (H\ot H\ot ((\mu_H\ot H)\co (H\ot c_{H,H})))\co ((\delta_H\co \eta_H)\ot (\delta_H\co \eta_H)\ot H)$

\item[ ]$=((\varepsilon_{H}\co \mu_H)\ot H)\co (H\ot c_{H,H})\co ((\mu_{H\ot H}\co (\delta_H\ot \delta_H)\ot H)\co (\eta_H\ot \eta_H\ot H)$

\item[ ]$=\Pi_{H}^{L}.$
\end{itemize}

The proof of (\ref{aux-1-monoid-hl}) and (\ref{aux-2-monoid-hl}) is the same that the given in \cite{AFG-Cleftwhq}, and the equalities (\ref{aux-1-monoid-hr}) and (\ref{aux-2-monoid-hr}) follow a similar pattern. As far as the last equalities, the proof is somewhat different to the given in \cite{AFG-Cleftwhq} because in this case we can not use the antipode. We only show (\ref{monoid-hl-2}), the other being analogous. Using that $H$ is a comonoid, condition (a1) of Definition \ref{Weak-Hopf-quasigroup} (twice), (\ref{aux-2-monoid-hl}), condition (a2) of Definition \ref{Weak-Hopf-quasigroup} and (\ref{aux-1-monoid-hl}),

\begin{itemize}
\item[ ]$\hspace{0.38cm} \mu_{H}\co (H\ot (\mu_{H}\co (\Pi_{H}^{L}\ot H)))$

\item[ ]$=(\varepsilon_H\ot H)\co \delta_H\co \mu_{H}\co (H\ot (\mu_{H}\co (\Pi_{H}^{L}\ot H)))$

\item[ ]$=(\varepsilon_H\ot H)\co \mu_{H\ot H}\co ((\delta_H\co \mu_H\co (H\ot \Pi_{H}^{L}))\ot \delta_H)$

\item[ ]$=((\varepsilon_H\co \mu_{H}\co (\mu_H\ot H))\ot \mu_H)\co (H\ot H\ot c_{H,H}\ot H)\co (H\ot c_{H,H}\ot H\ot H)\co (\delta_H\ot \Pi_{H}^{L} \ot \delta_H)$

\item[ ]$=(\varepsilon_H\ot H)\co \mu_{H\ot H}\co (\delta_H\ot ((\mu_H\ot H)\co (\Pi_{H}^{L}\ot \delta_H)))$

\item[ ]$=(\varepsilon_H\ot H)\co \mu_{H\ot H}\co (\delta_H\ot \delta_H)\co (H\ot (\mu_H\co (\Pi_{H}^{L}\ot H)))$

\item[ ]$=\mu_{H}\co ((\mu_{H}\co (H\ot \Pi_{H}^{L}))\ot H).$
\end{itemize}

\end{proof}

\begin{remark}
\label{HsubLmonoide}
{\rm Let $H$ be a unital magma and comonoid such that conditions (a1), (a2) and (a3) of Definition \ref{Weak-Hopf-quasigroup} hold.
 Denote by $H_L=Im (\Pi_{H}^{L})$ and let $p_{L}:H\rightarrow H_L$, $i_{L}:H_L\rightarrow H$ be the morphisms such that
 $i_L\co p_L=\Pi_{H}^{L}$ and $p_{L}\co i_{L}=id_{H_L}$. 
Then the equalities (\ref{monoid-hl-1}), (\ref{monoid-hl-2})
 and (\ref{monoid-hl-3}) imply that $(H_L, \eta_{H_L}=p_L\co \eta_H, \mu_{H_L}=p_L\co \mu_H\co (i_L\ot i_L))$ is a monoid.
 Therefore we can consider the category of right (left) $H_L$-modules, denoted by $\mathcal C_{H_L}$ ($_{H_L}\mathcal C$).
 In particular, $(H, \phi_{H}^{L}=\mu_H\co (H\ot i_L))$ is in $\mathcal C_{H_L}$ and $(H, \varphi_{H}^{L}=\mu_H\co (i_L\ot H))$ is in $_{H_L}\mathcal C$.
Moreover, if $(M \phi_M)$ is a right $H_L$-module, we define $M\ot_{H_L}H$ by the following coequalizer diagram:

$$
\setlength{\unitlength}{1mm}
\begin{picture}(101.00,10.00)
\put(22.00,8.00){\vector(1,0){40.00}}
\put(22.00,4.00){\vector(1,0){40.00}}
\put(75.00,6.00){\vector(1,0){21.00}}
\put(43.00,11.00){\makebox(0,0)[cc]{$M\ot \varphi_{H}^{L}$ }}
\put(43.00,0.00){\makebox(0,0)[cc]{$\phi_M\ot H$ }}
\put(85.00,9.00){\makebox(0,0)[cc]{$n_M$ }}
\put(10.00,6.00){\makebox(0,0)[cc]{$ M\ot H_L\ot H$ }}
\put(70.00,6.00){\makebox(0,0)[cc]{$M\ot H$ }}
\put(105.00,6.00){\makebox(0,0)[cc]{$M\ot_{H_L}H$ }}
\end{picture}
$$

Similar considerations can be done for $H_R=Im (\Pi_{H}^{R})$.
}
\end{remark}

Now let $H$ be a unital magma and comonoid such that conditions (a1), (a2) and (a3) of Definition \ref{Weak-Hopf-quasigroup} hold. Define the morphisms, called $\Omega$-morphisms, 
$$\Omega_{L}^{1}=(\mu_H\ot H)\co (H\ot \Pi_{H}^{L}\ot H)\co (H\ot \delta_H),$$
$$\Omega_{R}^{1}=(\mu_H\ot H)\co (H\ot \Pi_{H}^{R}\ot H)\co (H\ot \delta_H),$$
$$\Omega_{L}^{2}=(H\ot \mu_H)\co (H\ot \Pi_{H}^{L}\ot H)\co (\delta_H\ot H)$$
and
$$\Omega_{R}^{2}=(H\ot \mu_H)\co (H\ot \Pi_{H}^{R}\ot H)\co (\delta_H\ot H).$$

By Proposition \ref{otherproperties}, it is not difficult to see that these morphisms are idempotent. As a consequence, there exist objects $H\times_{L}^{1} H$,
 $H\times_{R}^{1} H$, $H\times_{L}^{2} H$ and $H\times_{R}^{2} H$ and morphisms
$$q_{L}^{1}:H\ot H \rightarrow H\times_{L}^{1} H\;\;\;, j_{L}^{1}:H\times_{L}^{1} H\rightarrow H\ot H,$$
$$q_{R}^{1}:H\ot H \rightarrow H\times_{R}^{1} H\;\;\;, j_{R}^{1}:H\times_{R}^{1} H\rightarrow H\ot H,$$
$$q_{L}^{2}:H\ot H \rightarrow H\times_{L}^{2} H\;\;\;, j_{L}^{2}:H\times_{L}^{2} H\rightarrow H\ot H,$$
and
$$q_{R}^{2}:H\ot H \rightarrow H\times_{R}^{2} H\;\;\;, j_{R}^{2}:H\times_{R}^{2} H\rightarrow H\ot H,$$
such that, for $\sigma\in \{L,R\}$ and $\alpha \in \{1,2\}$,

\begin{equation}
\label{Omegas} j_{\sigma}^{\alpha}\co q_{\sigma}^{\alpha}=\Omega_{\sigma}^{\alpha} ,\;\;\; q_{\sigma}^{\alpha}\co j_{\sigma}^{\alpha}=id_{H\times_{\sigma}^{\alpha}H}.
\end{equation}

Finally, by conditions (\ref{monoid-hl-1}) and (\ref{monoid-hl-3}), it is easy to see that
\begin{equation}
\label{muconomega} 
(\mu_H\ot H)\co (H\ot \Omega_{\sigma}^{1})=(H\ot \Omega_{\sigma}^{1})\co (\mu_H\ot H)
\end{equation}
and
\begin{equation}
\label{omegaconmu}
(H\ot \mu_H)\co (\Omega_{\sigma}^{2}\ot H)=(\Omega_{\sigma}^{2}\ot H)\co (H\ot\mu_H), \sigma\in \{L,R\}.
\end{equation}

Note that the morphism $\Omega_{R}^{1}$ is the same that the one defined in \cite{AFG-Cleftwhq} by the name of $\nabla_H$. The following Lemma gives an explanation of the meaning of the objects 
$H\times_{L}^{1} H$, $H\times_{R}^{1} H$, $H\times_{L}^{2} H$ and $H\times_{R}^{2} H$ by using equalizer and coequalizer diagrams.

\begin{lemma}
\label{diagrams}
Let $H$ be a unital magma and comonoid such that conditions (a1), (a2) and (a3) of Definition \ref{Weak-Hopf-quasigroup} hold. Then we have that:
 
\begin{itemize}
\item[(i)] The diagrams
$$
\setlength{\unitlength}{1mm}
\begin{picture}(101.00,10.00)
\put(22.00,8.00){\vector(1,0){40.00}}
\put(22.00,4.00){\vector(1,0){40.00}}
\put(75.00,6.00){\vector(1,0){21.00}}
\put(43.00,11.00){\makebox(0,0)[cc]{$\phi_{H}^{L}\ot H$ }}
\put(43.00,0.00){\makebox(0,0)[cc]{$H\ot \varphi_{H}^{L}$ }}
\put(85.00,9.00){\makebox(0,0)[cc]{$q_{L}^{1}$ }}
\put(10.00,6.00){\makebox(0,0)[cc]{$ H\ot H_L\ot H$ }}
\put(70.00,6.00){\makebox(0,0)[cc]{$H\ot H$ }}
\put(105.00,6.00){\makebox(0,0)[cc]{$H\times_{L}^{1} H$ }}
\end{picture}
$$
and
$$
\setlength{\unitlength}{1mm}
\begin{picture}(101.00,10.00)
\put(22.00,8.00){\vector(1,0){40.00}}
\put(22.00,4.00){\vector(1,0){40.00}}
\put(75.00,6.00){\vector(1,0){21.00}}
\put(43.00,11.00){\makebox(0,0)[cc]{$\phi_{H}^{R}\ot H$ }}
\put(43.00,0.00){\makebox(0,0)[cc]{$H\ot \varphi_{H}^{R}$ }}
\put(85.00,9.00){\makebox(0,0)[cc]{$q_{R}^{2}$ }}
\put(10.00,6.00){\makebox(0,0)[cc]{$ H\ot H_R\ot H$ }}
\put(70.00,6.00){\makebox(0,0)[cc]{$H\ot H$ }}
\put(105.00,6.00){\makebox(0,0)[cc]{$H\times_{R}^{2} H$ }}
\end{picture}
$$
are coequalizer diagrams.
By Remark \ref{HsubLmonoide}, we have that $H\times_{L}^{1} H\cong H\ot_{H_L}H$ and $H\times_{R}^{2} H\cong H\ot_{H_R}H$.

\item[(ii)] The diagrams

$$
\setlength{\unitlength}{3mm}
\begin{picture}(30,4)
\put(3,2){\vector(1,0){4}}
\put(11,2.5){\vector(1,0){15}}
\put(11,1.5){\vector(1,0){15}}
\put(1,2){\makebox(0,0)[cc]{$H\times_{L}^{2} H$}}
\put(9,2){\makebox(0,0)[cc]{$H\ot H$}}
\put(30,2){\makebox(0,0)[cc]{$H\ot H_L\ot H$}}
\put(5.5,3){\makebox(0,0)[cc]{$j_{L}^{2}$}}
\put(19,3.5){\makebox(0,0)[cc]{$((H\ot p_L)\co \delta_{H})\ot H$}}
\put(19,0.5){\makebox(0,0)[cc]{$H\ot ((p_L\ot H)\co\delta_{H})$}}
\end{picture}
$$
and
$$
\setlength{\unitlength}{3mm}
\begin{picture}(30,4)
\put(3,2){\vector(1,0){4}}
\put(11,2.5){\vector(1,0){15}}
\put(11,1.5){\vector(1,0){15}}
\put(1,2){\makebox(0,0)[cc]{$H\times_{R}^{1} H$}}
\put(9,2){\makebox(0,0)[cc]{$H\ot H$}}
\put(30,2){\makebox(0,0)[cc]{$H\ot H_L\ot H$}}
\put(5.5,3){\makebox(0,0)[cc]{$j_{R}^{1}$}}
\put(19,3.5){\makebox(0,0)[cc]{$((H\ot p_R)\co \delta_{H})\ot H$}}
\put(19,0.5){\makebox(0,0)[cc]{$H\ot ((p_R\ot H)\co\delta_{H})$}}
\end{picture}
$$
are equalizer diagrams.
\end{itemize}

\end{lemma}

\begin{proof}

$(i).$ We will give the computations for the first diagram, the proof for the other is similar. First of all,

\begin{itemize}
\item[ ]$\hspace{0.38cm} \Omega_{L}^{1}\co (H\ot \varphi_{H}^{L})$
\item[ ]$=((\mu_H\co (H\ot \Pi_{H}^{L}))\ot H)\co (H\ot (\delta_H\co \mu_H\co (i_L\ot H)))$
\item[ ]$=((\mu_H\co (H\ot \Pi_{H}^{L}))\ot H)\co (H\ot ((\mu_H\ot H)\co (i_L\ot \delta_H)))$
\item[ ]$=((\varepsilon_H\co \mu_H\co (H\ot \mu_H))\ot H\ot H)\co (H\ot H\ot c_{H,H}\ot H)\co (H\ot c_{H,H}\ot \delta_H)\co (\delta_H\ot i_L\ot H)$
\item[ ]$=((\varepsilon_H\co \mu_H\co (\mu_H\ot H))\ot H\ot H)\co (H\ot H\ot c_{H,H}\ot H)\co (H\ot c_{H,H}\ot \delta_H)\co (\delta_H\ot i_L\ot H)$
\item[ ]$=((\varepsilon_H\co \mu_H)\ot H\ot H)\co \delta_{H\ot H}\co ((\mu_H\co (H\ot i_L))\ot H)$
\item[ ]$=\Omega_{L}^{1}\co (\phi_{H}^{L}\ot H),$
\end{itemize}
where the first equality follows by the definition of $\Omega_{L}^{1}$; the second one by (\ref{aux-1-monoid-hl}); in the third and the last ones we use (\ref{mu-pi-l}); the fourth one relies on (a2) of Definition \ref{Weak-Hopf-quasigroup}; finally, the fifth equality follows by (\ref{aux-2-monoid-hl}).

By composing on the left with $q_{L}^{1}$, we have that $q_{L}^{1}\co (H\ot \varphi_{H}^{L})=q_{L}^{1}\co (\phi_{H}^{L}\ot H)$. Now assume that $r:H\ot H\rightarrow Q$ is a morphism such that $r\co (H\ot \varphi_{H}^{L})=r\co (\phi_{H}^{L}\ot H)$. Then the morphism
 $r\co j_{L}^{1}:H\times_{L}^{1} H\rightarrow Q$ satisfies that
$$r\co j_{L}^{1}\co q_{L}^{1}=r\co \Omega_{L}^{1}=r\co ((\mu_H\co (H\ot (i_L\co p_L)))\ot H)\co (H\ot \delta_H)=r\co (H\ot (\Pi_{H}^{L}\ast id_H))=r,$$
and if $s:H\ot H\rightarrow Q$ is such that $s\co q_{L}^{1}=r$, then $s=s\co q_{L}^{1}\co j_{L}^{1}=r\co j_{L}^{1}$.

$(ii).$ We only give the computations for the first diagram. Composing on the right with $q_{L}^{2}$ and on the left whit $H\ot i_L\ot H$,  
\begin{itemize}
\item[ ]$\hspace{0.38cm} (H\ot \Pi_{H}^{L}\ot H)\co (H\ot \delta_H)\co \Omega_{L}^{2}$
\item[ ]$=(H\ot \Pi_{H}^{L}\ot H)\co (H\ot (\delta_H\co \mu_H\co (\Pi_{H}^{L}\ot H)))\co (\delta_H\ot H)$
\item[ ]$=(H\ot \Pi_{H}^{L}\ot H)\co (H\ot ((\mu_H\ot H)\co (\Pi_{H}^{L}\ot \delta_H)))\co (\delta_H\ot H)$
\item[ ]$=(H\ot (\Pi_{H}^{L}\co \mu_H\co (\Pi_{H}^{L}\ot \overline{\Pi}_{H}^{L}))\ot H)\co(\delta_H\ot \delta_H)$
\item[ ]$=(H\ot (\Pi_{H}^{L}\co \mu_H)\ot H)\co (H\ot \Pi_{H}^{L}\ot ((H\ot \mu_H)\co ((\delta_H\co \eta_H)\ot H)))\co (\delta_H\ot H)$
\item[ ]$=(H\ot \Pi_{H}^{L}\ot \mu_H)\co (H\ot ((H\ot \overline{\Pi}_{H}^{R})\co \delta_H\co \Pi_{H}^{L})\ot H)\co (\delta_H\ot H)$
\item[ ]$=(H\ot \Pi_{H}^{L}\ot \mu_H)\co (H\ot (\delta_H\co \Pi_{H}^{L})\ot H)\co (\delta_H\ot H)$
\item[ ]$=(H\ot \Pi_{H}^{L}\ot \mu_H)\co (H\ot (((\varepsilon_H\co \mu_H)\ot H)\co (H\ot c_{H,H})\co (\delta_H\ot H))\ot H\ot H)$
\item[ ]$\hspace{0.38cm} \co (H\ot H\ot c_{H,H}\ot H)\co (H\ot (\delta_H\co \eta_H)\ot H\ot H)\co (\delta_H\ot H)$
\item[ ]$=(H\ot (\Pi_{H}^{L}\co \mu_H)\ot \mu_H)\co (H\ot H\ot c_{H,H}\ot H)\co (H\ot (\delta_H\co \eta_H)\ot H\ot H)\co (\delta_H\ot H)$

\item[ ]$=(H\ot \Pi_{H}^{L}\ot \mu_H)\co (H\ot ((H\ot\Pi_{H}^{L})\co \delta_H)\ot H)\co (\delta_H\ot H)$
\item[ ]$=(H\ot \Pi_{H}^{L}\ot H)\co (\delta_H\ot H)\co \Omega_{L}^{2},$
\end{itemize}
where the first equality follows by the definition of $\Omega_{L}^{2}$; the second one by (\ref{aux-1-monoid-hl}); in the third one we use (\ref{pi-l-mu-pi-l}); the fourth equality relies on (\ref{pi-l-barra-delta}); the fifth one follows by (\ref{pi-r-barra-delta}); the sixth one by (\ref{pi-delta-mu-pi-3}); the seventh one uses coassociativity and the definition of $\Pi_{H}^{L}$, the eighth one relies on (\ref{mu-pi-l}); the ninth one uses (\ref{delta-pi-l}); finally, the last one follows by coassociativity.

As a consequence, $(H\ot ((p_L\ot H)\co \delta_H))\co j_{L}^{2}=(((H\ot p_L)\co \delta_H)\ot H)\co j_{L}^{2}$ and, if $r:Q\rightarrow H\ot H$ is a morphism such that
$(H\ot ((p_L\ot H)\co \delta_H))\co r=(((H\ot p_L)\co \delta_H)\ot H)\co r$, it is easy to see that the morphism $q_{L}^{2}\co r$ satisfies that $j_{L}^{2}\co q_{L}^{2}\co r=r$, and it is unique because, if $s:Q\rightarrow H\ot H$ is such that $j_{L}^{2}\co s=r$, then $s=q_{L}^{2}\co j_{L}^{2}\co s=q_{L}^{2}\co r$.
\end{proof}

The following definition is inspired in \cite{Brz}.

\begin{definition}
\label{almostlinear}
{\rm Let $H$ be a magma. We say that a morphism $\phi:H\ot H\rightarrow H\ot H$ is:
\begin{itemize}
\item[(i)] Almost left $H$-linear, if $\phi=(\mu_H\ot H)\co (H\ot \phi)\co (H\ot \eta_H\ot H)$.
\item[(ii)] Almost right $H$-linear, if $\phi=(H\ot \mu_H)\co (\phi\ot H)\co (H\ot \eta_H\ot H)$.
\end{itemize}
By dualization, if $H$ is a comagma, we will say that a morphism $\phi$ is almost left $H$-colinear if $\phi=(H\ot \varepsilon_H\ot H)\co (H\ot \phi)\co (\delta_H\ot H)$, and almost right $H$-colinear if $\phi=(H\ot \varepsilon_H\ot H)\co (\phi\ot H)\co (H\ot \delta_H)$.
}
\end{definition}

\begin{proposition}
\label{examplesalmost}
Let $H$ be a magma and comagma. The following assertions hold.
\begin{itemize}
\item[(i)] The right Galois morphism, defined as $\beta=(\mu_H\ot H)\co (H\ot \delta_H)$ is almost left $H$-linear and almost right $H$-colinear.
\item[(ii)] The left Galois morphism, defined as $\gamma=(H\ot \mu_H)\co (\delta_H\ot H)$ is almost right $H$-linear and almost left $H$-colinear.
\item[(iii)] The morphisms $\Omega_{L}^{1}$ and $\Omega_{R}^{1}$ are almost left $H$-linear and almost right $H$-colinear.
\item[(iv)] The morphisms $\Omega_{L}^{2}$ and $\Omega_{R}^{2}$ are almost right $H$-linear and almost left $H$-colinear.
\end{itemize}
Moreover, if $H$ is a unital magma and comonoid such that conditions (a1), (a2) and (a3) of Definition \ref{Weak-Hopf-quasigroup} hold:
\begin{itemize}
\item[(v)] The morphism $\Omega_{L}^{1}$ is almost right $H$-linear and it is almost left $H$-colinear if and only if $\Pi_{H}^{L}=\overline{\Pi}_{H}^{L}$.
\item[(vi)] The morphism $\Omega_{R}^{1}$ is almost left $H$-colinear and it is almost right $H$-linear if and only if $\overline{\Pi}_{H}^{L}=\Pi_{H}^{R}$.
\item[(vii)] The morphism $\Omega_{L}^{2}$ is almost right $H$-colinear and it is almost left $H$-linear if and only if $\Pi_{H}^{L}=\overline{\Pi}_{H}^{R}$.
\item[(viii)] The morphism $\Omega_{R}^{2}$ is almost left $H$-linear and it is almost right $H$-linear if and only if $\Pi_{H}^{R}=\overline{\Pi}_{H}^{R}$.
\end{itemize}
\end{proposition}

\begin{proof}

It is easy to see assertions (i)-(iv). As far as (v), we get the almost right $H$-linearity by using (\ref{pi-l-barra-delta}) and (\ref{pi-composition-2}). Indeed,
$$(H\ot \mu_H)\co (\Omega_{L}^{1}\ot H)\co (H\ot \eta_H\ot H)=(\mu_H\ot H)\co (H\ot (\Pi_{H}^{L}\co\overline{\Pi}_{H}^{L})\ot H)\co (H\ot \delta_H)=\Omega_{L}^{1}.$$
On the other hand, using (\ref{mu-pi-l}) and (\ref{mu-pi-l-var}),
$$(H\ot \varepsilon_H\ot H)\co (H\ot \Omega_{L}^{1})\co (\delta_H\ot H)=(H\ot (\varepsilon_H\co \mu_H)\ot H)\co (\delta_H\ot \delta_H)=
(\mu_H\ot H)\co (H\ot \overline{\Pi}_{H}^{L}\ot H)\co (H\ot \delta_H),$$
and as a consequence we have that $\Omega_{L}^{1}$ is almost left $H$-colinear if and only if $\Pi_{H}^{L}=\overline{\Pi}_{H}^{L}$.

To get (vi), the morphism $\Omega_{R}^{1}$ is almost left $H$-colinear because by (\ref{mu-pi-l-var}) and (\ref{pi-composition-2}),
\begin{itemize}
\item[ ]$\hspace{0.38cm} (H\ot \varepsilon_H\ot H)\co (H\ot \Omega_{R}^{1})\co (\delta_H\ot H)=(H\ot (\varepsilon_H\co \mu_H)\ot H)\co (\delta_H\ot ((\Pi_{H}^{R}\ot H)\co\delta_H))$
\item[ ]$=(\mu_H\ot H)\co (H\ot (\overline{\Pi}_{H}^{L}\co \Pi_{H}^{R})\ot H)\co (H\ot \delta_H)=\Omega_{R}^{1}.$
\end{itemize}

Moreover, using (\ref{delta-pi-r}) and (\ref{pi-l-barra-delta}),
$$(H\ot \mu_H)\co (\Omega_{R}^{1}\ot H)\co (H\ot \eta_H\ot H)=(\mu_H\ot \mu_H)\co (H\ot (\delta_H\co \eta_H)\ot H)=(\mu_H\ot H)\co (H\ot \overline{\Pi}_{H}^{L}\ot H)\co (H\ot \delta_H),$$
and then $\Omega_{R}^{1}$ is almost right $H$-linear if and only if $\overline{\Pi}_{H}^{L}=\Pi_{H}^{R}$. We leave to the reader the proofs for (vii) and (viii). 
\end{proof}

\begin{remark}
{\rm
Note that, as we showed in Propositions 1.5. and 1.6. of \cite{AFGLV} (the proofs do not use associativity nor the antipode), 
$\Pi_{H}^{L}=\overline{\Pi}_{H}^{L}$ iff $\Pi_{H}^{R}=\overline{\Pi}_{H}^{R}$, and $\overline{\Pi}_{H}^{L}=\Pi_{H}^{R}$ iff $\Pi_{H}^{R}=\overline{\Pi}_{H}^{L}$.
Therefore, the morphism $\Omega_{L}^{1}$ is almost left $H$-colinear if and only if $\Omega_{R}^{2}$ is almost right $H$-linear (that is the case, for example, if $H$ is coconmutative, i.e., $\delta_H=c_{H,H}\co \delta_H$), and the morphism $\Omega_{R}^{1}$ is almost right $H$-linear if and only if $\Omega_{L}^{2}$ is almost left $H$-linear (for example, if $H$ is commutative, i. e., $\mu_H=\mu_H\co c_{H,H}$).
}
\end{remark}

\begin{remark}
{\rm Note that, if $H$ is a weak Hopf quasigroup, we can express the $\Omega$-morphisms as compositions of the Galois maps. Actually, by (a4) we have that
\begin{equation}
\label{equalitiesomega}
\Omega_{L}^{1}=\overline{\beta}\co \beta ,\;\;\;\Omega_{R}^{1}=\beta\co\overline{\beta} ,\;\;\;\Omega_{L}^{2}=\gamma\co\overline{\gamma},\;\;\;\Omega_{R}^{2}=\overline{\gamma}\co \gamma,  
\end{equation}
where $\overline{\beta}=(\mu_H\ot H)\co (H\ot \lambda_H\ot H)\co (H\ot \delta_H)$ and $\overline{\gamma}=(H\ot\mu_H)\co (H\ot \lambda_H\ot H)\co (\delta_H\ot H)$.
Moreover, if the weak Hopf quasigroup $H$ is a Hopf quasigroup,  $\Pi_{H}^{L}=\Pi_{H}^{R}=\overline{\Pi}_{H}^{L}=\overline{\Pi}_{H}^{R}=\varepsilon_H\ot\eta_H$ and then the $\Omega$-morphism are identities. As a consequence we have that in this case the Galois maps $\beta$ and $\gamma$ are isomorphisms with inverses $\overline{\beta}$ and $\overline{\gamma}$, respectively.
}
\end{remark}

Now we give the main result of this paper, which characterizes weak Hopf quasigrous in terms of a composition involving the Galois maps.

\begin{theorem}
\label{characterization}
Let $H$ be a unital magma and comonoid such that conditions (a1), (a2) and (a3) of Definition \ref{Weak-Hopf-quasigroup} hold. The following assertions are equivalent:

\begin{itemize}
\item[(i)] $H$ is a weak Hopf quasigroup.
\item[(ii)] The morphisms $f=q_{R}^{1}\co \beta\co j_{L}^{1}:H\times_{L}^{1} H\rightarrow H\times_{R}^{1} H$ and $g=q_{L}^{2}\co \gamma\co j_{R}^{2}:H\times_{R}^{2} H\rightarrow H\times_{L}^{2} H$ are isomorphisms, the morphism $j_{L}^{1}\co f^{-1}\co q_{R}^{1}$ is almost left $H$-linear, and $j_{R}^{2}\co g^{-1}\co q_{L}^{2}$ is almost right $H$-linear.
\end{itemize}

\end{theorem}

\begin{proof}
$(i)\Rightarrow(ii).$ Assume that $H$ is a weak Hopf quasigroup. Define $f^{-1}=q_{L}^{1}\co\overline{\beta}\co j_{R}^{1}$ and $g^{-1}=q_{R}^{2}\co\overline{\gamma}\co j_{L}^{2}$. Then $f^{-1}$ and $g^{-1}$ are the inverses of $f$ and $g$, respectively. Indeed,

$$f\co f^{-1}=q_{R}^{1}\co \beta\co \Omega_{L}^{1}\co\overline{\beta}\co j_{R}^{1}
=q_{R}^{1}\co \beta\co \overline{\beta}\co \beta\co \overline{\beta}\co j_{R}^{1}=q_{R}^{1}\co \Omega_{R}^{1}\co \Omega_{R}^{1}\co j_{R}^{1}
=q_{R}^{1}\co \Omega_{R}^{1}\co j_{R}^{1}=q_{R}^{1}\co j_{R}^{1}\co q_{R}^{1}\co j_{R}^{1}=id_{H\times_{R}^{1} H}.$$

On the other hand,

$$ f^{-1}\co f=q_{L}^{1}\co \overline{\beta}\co \Omega_{R}^{1}\co\beta\co j_{L}^{1}
=q_{L}^{1}\co \overline{\beta}\co \beta\co \overline{\beta}\co \beta\co j_{L}^{1}
=q_{L}^{1}\co \Omega_{L}^{1}\co \Omega_{L}^{1}\co j_{L}^{1}=q_{L}^{1}\co \Omega_{L}^{1}\co j_{L}^{1}
=q_{L}^{1}\co j_{L}^{1}\co q_{L}^{1}\co j_{L}^{1}=id_{H\times_{L}^{1} H},$$

and then $f^{-1}$ is the inverse of $f$. In a similar way it is easy to see that $g^{-1}$ is the inverse of $g$.
To see the almost left and right $H$-linearity, we will see that $j_{L}^{1}\co f^{-1}\co q_{R}^{1}=\overline{\beta}$ and $j_{R}^{2}\co g^{-1}\co q_{R}^{2}=\overline{\gamma}$. We only show the first equality, the second one follows a similar pattern. Indeed, using the definition of $f^{-1}$, the idempotent character of $\Omega_{L}^{1}$, equality (\ref{monoid-hl-2}) for $\Pi_{H}^{R}$, coassociativity and (a4-3) of Definition \ref{Weak-Hopf-quasigroup}, we obtain that
\begin{itemize}
\item[ ]$\hspace{0.38cm} j_{L}^{1}\co f^{-1}\co q_{R}^{1}=\Omega_{L}^{1}\co \overline{\beta}\co \Omega_{R}^{1}
=\overline{\beta}\co \beta\co\overline{\beta}\co \beta\co\overline{\beta}=\overline{\beta}\co \beta\co\overline{\beta}=\overline{\beta}\co \Omega_{R}^{1}$
\item[ ]$=((\mu_H\co (\mu_H\ot H)\co (H\ot \Pi_{H}^{R}\ot H))\ot H)\co (H\ot H\ot ((\lambda_H\ot H)\co \delta_H))\co (H\ot \delta_H)$
\item[ ]$=(\mu_H\ot H)\co (H\ot (\Pi_{H}^{R}\ast \lambda_H)\ot H)\co (H\ot \delta_H)=\overline{\beta}.$
\end{itemize}
$(ii)\Rightarrow (i).$ First of all, note that 
\begin{equation}
\label{betaandgammaexpressions}
j_{R}^{1}\co f\co q_{L}^{1}=\beta\;\;\; j_{L}^{2}\co g\co q_{R}^{2}=\gamma .
\end{equation}

Indeed,
\begin{itemize}
\item[ ]$\hspace{0.38cm} j_{R}^{1}\co f\co q_{L}^{1}$
\item[ ]$=(\mu_H\ot H)\co (H\ot \Pi_{H}^{R}\ot H)\co (H\ot \delta_H)\co (\mu_H\ot H)\co (H\ot \delta_H)\co (\mu_H\ot H)\co (H\ot \Pi_{H}^{L}\ot H)\co (H\ot \delta_H)$
\item[ ]$=(\mu_H\ot H)\co (H\ot \Pi_{H}^{R}\ot H)\co (H\ot \delta_H)\co (\mu_H\ot H)\co (H\ot (\Pi_{H}^{L}\ast id_H)\ot H)\co (H\ot \delta_H)$
\item[ ]$=(\mu_H\ot H)\co (H\ot \Pi_{H}^{R}\ot H)\co (H\ot \delta_H)\co (\mu_H\ot H)\co (H\ot \delta_H)$
\item[ ]$=(\mu_H\ot H)\co (H\ot (id_H\ast \Pi_{H}^{R})\ot H)\co (H\ot \delta_H)$
\item[ ]$=\beta,$
\end{itemize}
where the first equalitiy follows by the definitions of $f$, $\Omega_{L}^{1}$ and $\Omega_{R}^{1}$; the second and the fourth ones by (\ref{monoid-hl-2}) and (\ref{monoid-hl-3}); and in the third and the last ones we use (a4-3) of Definition \ref{Weak-Hopf-quasigroup}.
In a similar way, we get the second equality. As a consequence, we obtain the following expressions for $\mu_H$ and $\delta_H$:
\begin{equation}
\label{muexpression}
\mu_H=(H\ot \varepsilon_H)\co j_{R}^{1}\co f\co q_{L}^{1}=(\varepsilon_H\ot H)\co j_{L}^{2}\co g\co q_{R}^{2}.
\end{equation}
\begin{equation}
\label{deltaexpression}
\delta_H=j_{R}^{1}\co f\co q_{L}^{1}\co(\eta_H\ot H)=j_{L}^{2}\co g\co q_{R}^{2}\co (H\ot \eta_H).
\end{equation}
Now define $\lambda_H=(H\ot \varepsilon_H)\co j_{L}^{1}\co f^{-1}\co q_{R}^{1}\co (\eta_H\ot H)$ and $\overline{\lambda_H}=(\varepsilon_H\ot H)\co j_{R}^{2}\co g^{-1}\co q_{L}^{2} \co (H\ot \eta_H)$. To obtain that $H$ is a weak Hopf quasigroup we will show that $\lambda_H=\overline{\lambda_H}$ and they satisfy (a4) of Definition \ref{Weak-Hopf-quasigroup}. We begin showing that $id_H\ast \lambda_H=\Pi_{H}^{L}$. Indeed, by the almost left $H$-linearity and (\ref{deltaexpression}),

\begin{itemize}
\item[ ]$\hspace{0.38cm} id_H\ast \lambda_H=(H\ot \varepsilon_H)\co (\mu_H\ot H)\co (H\ot (j_{L}^{1}\co f^{-1}\co q_{R}^{1}\co (\eta_H\ot H)))\co \delta_H$
\item[ ]$=(H\ot \varepsilon_H)\co j_{L}^{1}\co f^{-1}\co q_{R}^{1}\co \delta_H=(H\ot \varepsilon_H)\co j_{L}^{1}\co f^{-1}\co q_{R}^{1}\co j_{R}^{1}\co f\co q_{L}^{1}\co(\eta_H\ot H)$
\item[ ]$=(H\ot \varepsilon_H)\co \Omega_{L}^{1}\co (\eta_H\ot H)=\Pi_{H}^{L}.$
\end{itemize}
In a similar way, but using the almost right $H$-linearity, we get that $\overline{\lambda_H}\ast id_H=\Pi_{H}^{R}$.
On the other hand, note that $(\beta\ot H)\co (H\ot \delta_H)=(H\ot \delta_H)\co \beta$ holds and, by (\ref{betaandgammaexpressions}) it is easy to see that
\begin{equation}
\label{betaequality}
(H\ot \delta_H)\co j_{L}^{1}\co f^{-1}\co q_{R}^{1}=((j_{L}^{1}\co f^{-1}\co q_{R}^{1})\ot H)\co (H\ot \delta_H).
\end{equation}
Moreover, taking into account that $(H\ot\gamma)\co (\delta_H\ot H)=(\delta_H\ot H)\co \gamma$, we get
\begin{equation}
\label{gammaequality}
(\delta_H\ot H)\co j_{R}^{2}\co g^{-1}\co q_{L}^{2}=(H\ot (j_{R}^{2}\co g^{-1}\co q_{L}^{2}))\co (\delta_H\ot H).
\end{equation}
Therefore, using (\ref{muexpression}), we obtain that
\begin{itemize}
\item[ ]$\hspace{0.38cm} \lambda_H\ast id_H=\mu_H\co (H\ot ((\varepsilon_H\ot H)\co \delta_H)\co j_{L}^{1}\co f^{-1}\co q_{R}^{1}\co (\eta_H\ot H)$
\item[ ]$=(H\ot \varepsilon_H)\co j_{R}^{1}\co f\co q_{L}^{1}\co j_{L}^{1}\co f^{-1}\co q_{R}^{1}\co (\eta_H\ot H)
=(H\ot \varepsilon_H)\co \Omega_{R}^{1}\co (\eta_H\ot H)=\Pi_{H}^{R},$
\end{itemize}
and by similar computations, but using (\ref{gammaequality}), we have that $id_H\ast \overline{\lambda_H}=\Pi_{H}^{L}$.

To get (a4-3) of Definition \ref{Weak-Hopf-quasigroup}, 
$$\lambda_H\ast \Pi_{H}^{L}=\mu_H\co (H\ot \Pi_{H}^{L})\co j_{L}^{1}\co f^{-1}\co q_{R}^{1}\co (\eta_H\ot H)
=(H\ot \varepsilon_H)\co \Omega_{L}^{1}\co j_{L}^{1}\co f^{-1}\co q_{R}^{1}\co (\eta_H\ot H)=\lambda_H,$$

where the first equality follows by almost right $H$-linearity; the second one because $H$ is a comonoid; in the third one we use that $\mu_H\co (H\ot \Pi_{H}^{L})=(H\ot \varepsilon_H)\co \Omega_{L}^{1}$; finally, the last one follows because $\Omega_{L}^{1}\co j_{L}^{1}=j_{L}^{1}$. By similar computations but using almost left $H$-linearity and that $(\Pi_{H}^{L}\ot H)\co \delta_H=\Omega_{R}^{1}\co (\eta_H\ot H)$, it is not difficult to see that $\Pi_{H}^{R}\ast \lambda_H=\lambda_H$, and the same ideas can be used to show that $\overline{\lambda_H}= \overline{\lambda_H}\ast \Pi_{H}^{L}$.
Now we prove (a4-4)-(a4-7) of Definition \ref{Weak-Hopf-quasigroup}. Firstly, by almost right $H$-linearity and (\ref{betaandgammaexpressions})

\begin{itemize}
\item[ ]$\hspace{0.38cm} \mu_H\circ (\overline{\lambda_H}\ot \mu_H)\circ (\delta_H\ot H)
=(\varepsilon_H\ot H)\co ((H\ot \mu_H)\co ((j_{R}^{2}\co g^{-1}\co q_{L}^{2})\ot H)\co (H\ot \eta_H\ot H)\co \gamma$
\item[ ]$=(\varepsilon_H\ot H)\co j_{R}^{2}\co g^{-1}\co q_{L}^{2}\co j_{L}^{2}\co g\co q_{R}^{2}=(\varepsilon_H\ot H)\co \Omega_{R}^{2}
=\mu_{H}\co (\Pi_{H}^{R}\ot H),$
\end{itemize}

 and using almost right $H$-linearity, (\ref{gammaequality}) and (\ref{muexpression}),

\begin{itemize}
\item[ ]$\hspace{0.38cm} \mu_H\circ (H\ot \mu_H)\co (H\ot \overline{\lambda_H}\ot H)\co (\delta_H\ot H)
=\mu_H\co (H\ot \varepsilon_H\ot H)\co (H\ot (j_{R}^{2}\co g^{-1}\co q_{L}^{2}))\co (\delta_H\ot H)$
\item[ ]$=\mu_H\co (((H\ot\varepsilon_H)\co \delta_H)\ot H)\co j_{R}^{2}\co g^{-1}\co q_{L}^{2}
=(\varepsilon_H\ot H)\co j_{L}^{2}\co g\co q_{R}^{2} \co j_{R}^{2}\co g^{-1}\co q_{L}^{2}=(\varepsilon_H\ot H)\co \Omega_{L}^{2}$
\item[ ]$=\mu_{H}\co (\Pi_{H}^{L}\ot H).$
\end{itemize}

By similar ideas but using almost left $H$-linearity and (\ref{betaequality}), we show that

$$\mu_H\co (\mu_H\ot \lambda_H)\co (H\ot \delta_H)=\mu_{H}\co (H\ot \Pi_{H}^{L}),$$
and
$$\mu_H\co (\mu_H\ot H)\co (H\ot \lambda_H\ot H)\co (H\ot \delta_H)=\mu_{H}\co (H\ot \Pi_{H}^{R}).$$

To finish the proof, it only remains to see that $\lambda_H=\overline{\lambda_H}$. Indeed,

\begin{itemize}
\item[ ]$\hspace{0.38cm} \lambda_H=\lambda_H\ast \Pi_{H}^{L}
=\mu_H\co (H\ot \Pi_{H}^{L})\co (\lambda_H\ot H)\co \delta_H
=\mu_H\co (\mu_H\ot \lambda_H)\co (H\ot \delta_H)\co (\lambda_H\ot H)\co \delta_H$
\item[ ]$=\mu_H\co (\Pi_{H}^{R}\ot H)\co (H\ot \lambda_H)\co \delta_H=\mu_H\co (\overline{\lambda_H}\ot \mu_H)\co (\delta_H\ot H)\co (H\ot \lambda_H)\co \delta_H
=\overline{\lambda_H}\ast \Pi_{H}^{L}=\overline{\lambda_H},$
\end{itemize}

and the proof is complete.
\end{proof}

As we have said in the Introduction, the notion of weak Hopf quasigroup generalizes the ones of Hopf quasigroups and weak Hopf algebras. To finish this section we particularize our main theorem in these settings. Note that the first result is the assertion $(1)$ of the Theorem 2.5. (called the first fundamental theorem for Hopf (co)quasigroups) given by Brzezinski in \cite{Brz}.

\begin{corollary}
\label{corolarioHqg}
Let $H$ be a unital magma and comonoid such that $\varepsilon_H$ and $\delta_H$ are morphisms of unital magmas (equivalently, $\eta_H$ and $\mu_H$ are morphisms of counital comagmas). Then $H$ is a Hopf quasigroup if and only if the right and left Galois morphisms $\beta$ and $\gamma$ are isomorphisms and they have almost left $H$-linear and almost right $H$-linear inverses, respectively.
\end{corollary}
\begin{proof}
First of all, note that conditions (a2) and (a3) of Definition \ref{Weak-Hopf-quasigroup} trivialize because $\varepsilon_H$ and $\delta_H$ are morphisms of unital magmas. Moreover, $\Pi_{H}^{L}=\Pi_{H}^{R}=\overline{\Pi}_{H}^{L}=\overline{\Pi}_{H}^{R}=\varepsilon_H\ot\eta_H$ and then the $\Omega$-morphisms are identities. As a consequence, $f=\beta$ and $g=\gamma$.
\end{proof}

As far as weak Hopf algebras, we will prove that it is possible to remove the conditions about almost $H$-linearity. First we need to show the following technical Lemma:

\begin{lemma}
\label{fbaixaporvarphi}
Let $H$ be a unital magma and comonoid such that conditions (a1), (a2) and (a3) of Definition \ref{Weak-Hopf-quasigroup} hold. Let $f$ and $g$ be the maps defined in Theorem \ref{characterization} and define the morphisms:
\begin{equation}
\label{expressionsvarphi-1}
\varphi_{H\times_{R}^{1}H}=q_{R}^{1}\co (\mu_H\ot H)\co (H\ot j_{R}^{1}),\;\;\;\varphi_{H\times_{L}^{1} H}=q_{L}^{1}\co (\mu_H\ot H)\co (H\ot j_{L}^{1}),
\end{equation}
\begin{equation}
\label{expressionsvarphi-2}
\psi_{H\times_{R}^{2}H}=q_{R}^{2}\co (H\ot \mu_H)\co (j_{R}^{2}\ot H),\;\;\;\psi_{H\times_{L}^{2}H}=q_{L}^{2}\co (H\ot \mu_H)\co (j_{L}^{2}\ot H).
\end{equation}
 Then the following assertions are equivalent:
\begin{itemize}
\item[(i)] $H$ is a monoid.
\item[(ii)] The morphism $f$ satisfies that $f\co \varphi_{H\times_{L}^{1} H}=\varphi_{H\times_{R}^{1}H}\co (H\ot f)$.
\item[(iii)] The morphism $g$ satisfies that $g\co \psi_{H\times_{R}^{2} H}=\psi_{H\times_{L}^{2} H}\co (g\ot H)$.
\end{itemize}

\end{lemma}

\begin{proof}
$(i)\Rightarrow (ii).$  Assume that $H$ is a monoid. Then,

\begin{itemize}
\item[ ]$\hspace{0.38cm} f\co \varphi_{H\times_{L}^{1} H}=q_{R}^{1}\co \beta\co \Omega_{L}^{1}\co (\mu_H\ot H)\co (H\ot j_{L}^{1})=q_{R}^{1}\co \beta\co (\mu_H\ot H)\co (H\ot j_{L}^{1})$
\item[ ]$=q_{R}^{1}\co (\mu_H\ot H)\co (H\ot \beta)\co (H\ot j_{L}^{1})=q_{R}^{1}\co (\mu_H\ot H)\co (H\ot (\Omega_{R}^{1}\co\beta))\co (H\ot j_{L}^{1})=\varphi_{H\times_{R}^{1}H}\co (H\ot f),$
\end{itemize}

where the first and the last equalities are consequences of (\ref{expressionsvarphi-1}); the second and the fourth ones rely on (\ref{muconomega}). Finally, the third equality follows because $H$ is associative and then $(\mu_H\ot H)\co (H\ot \beta)=\beta\co (\mu_H\ot H)$.

To get $(ii)\Rightarrow (i)$, we will show that $(\mu_H\ot H)\co (H\ot \beta)=\beta\co (\mu_H\ot H)$. By composing with $H\ot \varepsilon_H$
 we obtain that $H$ is associative. First of all, note that, by (\ref{monoid-hl-2}) and (\ref{pi-l}),
$$\beta\co \Omega_{L}^{1}=(\mu_H\ot H)\co (H\ot (\Pi_{H}^{L}\ast id_H)\ot H)\co (H\ot \delta_H)=\beta,$$
and in a similar way but using (\ref{monoid-hl-2}) for $\Pi_{H}^{R}$ we get that $\Omega_{R}^{1}\co \beta=\beta$. Then,
 by (\ref{expressionsvarphi-1}) and (\ref{muconomega}),

$$ f\co \varphi_{H\times_{L}^{1} H}=q_{R}^{1}\co \beta\co \Omega_{L}^{1}\co (\mu_H\ot H)\co (H\ot j_{L}^{1})=q_{R}^{1}\co \beta\co (\mu_H\ot H)\co (H\ot j_{L}^{1})$$
and
$$ \varphi_{H\times_{R}^{1}H}\co (H\ot f)=q_{R}^{1}\co (\mu_H\ot H)\co (H\ot \Omega_{R}^{1})\co (H\ot (\beta\co j_{L}^{1}))=q_{R}^{1}\co (\mu_H\ot H)\co (H\ot (\beta\co j_{L}^{1})).$$

Composing with $j_{R}^{1}$ on the left and with $H\ot q_{L}^{1}$ on the right, and using (\ref{muconomega}) and (\ref{omegaconmu}) we obtain that $(\mu_H\ot H)\co (H\ot \beta)=\beta\co (\mu_H\ot H)$.

The proof for the equivalence between $(i)$ and $(iii)$ is similar and we leave the details to the reader.
\end{proof}

Now we can give our characterization for weak Hopf algebras. Note that the equivalence between $(i)$ and $(ii)$ is the result given by Schauenburg in \cite{S} (Theorem 6.1).
\begin{corollary}
\label{corolarioHweak}
Let $H$ be a monoid and comonoid such that conditions (a1), (a2) and (a3) of Definition \ref{Weak-Hopf-quasigroup} hold. The following assertions are equivalent.
\begin{itemize}
\item[(i)] $H$ is a weak Hopf algebra.
\item[(ii)] The morphism $f$ defined in Theorem \ref{characterization} is an isomorphism.
\item[(ii)] The morphism $g$ defined in Theorem \ref{characterization} is an isomorphism.
\end{itemize}
 \end{corollary}

\begin{proof}
By Theorem \ref{characterization}, $(i)\Rightarrow (ii)$ and $(i)\Rightarrow (iii)$. To get $(ii)\Rightarrow (i)$, we will begin by showing that, if $f$ is an isomorphism, the morphism $j_{L}^{1}\co f^{-1}\co q_{R}^{1}$ is always almost left $H$-linear. Indeed, note that by Lemma \ref{fbaixaporvarphi}, $f\co \varphi_{H\times_{L}^{1} H}=\varphi_{H\times_{R}^{1}H}\co (H\ot f)$. By the suitable compositions, we obtain that 
$\varphi_{H\times_{L}^{1} H}\co (H\ot f^{-1})=f^{-1}\co\varphi_{H\times_{R}^{1}H}$ and then
\begin{itemize}
\item[ ]$\hspace{0.38cm} (\mu_H\ot H)\co (H\ot (j_{L}^{1}\co f^{-1}\co q_{R}^{1}))\co (H\ot \eta_H\ot H)
=\Omega_{L}^{1}\co(\mu_H\ot H)\co (H\ot (j_{L}^{1}\co f^{-1}\co q_{R}^{1}))\co (H\ot \eta_H\ot H)$
\item[ ]$=j_{L}^{1}\co \varphi_{H\times_{L}^{1} H}\co (H\ot (f^{-1}\co q_{R}^{1}))\co (H\ot \eta_H\ot H)
=j_{L}^{1}\co f^{-1}\co\varphi_{H\times_{R}^{1}H}\co (H\ot q_{R}^{1})\co (H\ot \eta_H\ot H)$
\item[ ]$=j_{L}^{1}\co f^{-1}\co q_{R}^{1}\co (\mu_H\ot H)\co (H\ot \Omega_{R}^{1})\co (H\ot \eta_H\ot H)=j_{L}^{1}\co f^{-1}\co q_{R}^{1}\co\Omega_{R}^{1}=j_{L}^{1}\co f^{-1}\co q_{R}^{1}.$
\end{itemize}
Now we can follow the proof given in Theorem \ref{characterization} to see that the morphism $\lambda_H=(H\ot \varepsilon_H)\co j_{L}^{1}\co f^{-1}\co q_{R}^{1}\co (\eta_H\ot H)$ is the antipode of $H$ (in this case, by associativity of $H$, conditions (a4-4)-(a4-7) of Definition \ref{Weak-Hopf-quasigroup} trivialize).
The proof for $(iii)\Rightarrow (i)$ follows a similar pattern and we leave the details to the reader. 

\end{proof}

\section{A characterization for weak Hopf coquasigroups}
The notions of weak Hopf quasigroup and weak Hopf coquasigroup are entirely dual, i.e., we can obtain one of them by reversing arrows in the definition of the other. As a consequence, by dualizing the results given in the previous Section we get a characterization for weak Hopf coquasigroups. The proofs follow the same ideas, and in order to brevity they will be omitted.
First of all we introduce the notion of weak Hopf coquasigroup.

\begin{definition}
\label{Weak-Hopf-coquasigroup} {\rm  A weak Hopf coquasigroup $H$   in ${\mathcal
C}$ is a monoid $(H, \eta_H, \mu_H)$ and a counital comagma $(H,\varepsilon_H, \delta_H)$ such that the following axioms hold:
\begin{itemize}
\item[(b1)] $\delta_{H}\co \mu_{H}=(\mu_{H}\ot \mu_{H})\co \delta_{H\ot H}.$
\item[(b2)] $\varepsilon_{H}\co \mu_{H}\co (\mu_{H}\ot H)=((\varepsilon_{H}\co \mu_{H})\ot (\varepsilon_{H}\co \mu_{H}))\co (H\ot \delta_{H}\ot H)$ 
\item[ ]$=((\varepsilon_{H}\co \mu_{H})\ot (\varepsilon_{H}\co \mu_{H}))\co (H\ot (c_{H,H}^{-1}\co\delta_{H})\ot H).$
\item[b3)]$(\delta_{H}\ot H)\co \delta_{H}\co \eta_{H}=(H\ot \delta_{H})\co \delta_{H}\co \eta_{H}=(H\ot \mu_{H}\ot H)\co ((\delta_{H}\co \eta_{H}) \ot (\delta_{H}\co \eta_{H}))$
\item[ ]$=(H\ot (\mu_{H}\co c_{H,H}^{-1})\ot H)\co ((\delta_{H}\co \eta_{H}) \ot (\delta_{H}\co \eta_{H})).$
\item[(b4)] There exists $\lambda_{H}:H\rightarrow H$ in ${\mathcal C}$ (called the antipode of $H$) such that, if we denote by $\Pi_{H}^{L}$ (target morphism) and by $\Pi_{H}^{R}$ (source morphism) the morphisms 
$$\Pi_{H}^{L}=((\varepsilon_{H}\co\mu_{H})\ot H)\co (H\ot c_{H,H})\co ((\delta_{H}\co \eta_{H})\ot H),$$
$$\Pi_{H}^{R}=(H\ot(\varepsilon_{H}\co \mu_{H}))\co (c_{H,H}\ot H)\co (H\ot (\delta_{H}\co \eta_{H})),$$
then:
\begin{itemize}
\item[(b4-1)] $\Pi_{H}^{L}=id_{H}\ast \lambda_{H}.$
\item[(b4-2)] $\Pi_{H}^{R}=\lambda_{H}\ast id_{H}.$
\item[(b4-3)] $\lambda_{H}\ast \Pi_{H}^{L}=\Pi_{H}^{R}\ast \lambda_{H}= \lambda_{H}.$
\item[(b4-4)] $(\mu_H\ot H)\co (\lambda_H\ot \delta_H)\co \delta_H=(\Pi_{H}^{R}\ot H)\co \delta_H.$
\item[(b4-5)] $(\mu_H\ot H)\co (H\ot \lambda_H\ot H)\co (H\ot \delta_H)\co \delta_H=(\Pi_{H}^{L}\ot H)\co \delta_H.$
\item[(b4-6)] $(H\ot \mu_H)\co (\delta_H\ot \lambda_H)\co \delta_H=(H\ot \Pi_{H}^{L})\co \delta_H.$
\item[(b4-7)] $(H\ot \mu_H)\co (H\ot \lambda_H\ot H)\co (\delta_H\ot H)\co \delta_H=\mu_{H}\co (H\ot \Pi_{H}^{R}).$
\end{itemize}
\end{itemize}

Note that, if  $\eta_H$ and $\mu_H$ are  morphisms of counital comagmas, (equivalently, $\varepsilon_{H}$, $\delta_{H}$ are morphisms of unital magmas), $\Pi_{H}^{L}=\Pi_{H}^{R}=\eta_{H}\ot \varepsilon_{H}$ and, as a consequence, we have the notion of Hopf coquasigroup.
}
\end{definition}

Note that, when reversing arrows, the morphisms $\Pi_{H}^{L}$ and $\Pi_{H}^{R}$ are exactly the same of the previous section, while the morphism $\overline{\Pi}_{H}^{L}$ changes in $\overline{\Pi}_{H}^{R}$ and vice versa. As far as the $\Omega$-morphisms, we must change $\Omega_{L}^{1}$, $\Omega_{R}^{1}$, $\Omega_{L}^{2}$ and $\Omega_{R}^{2}$ by $\Omega_{L}^{2}$, $\Omega_{R}^{2}$, $\Omega_{L}^{1}$ and $\Omega_{R}^{1}$, respectively. Therefore the characterization of weak Hopf coquasigroups is the given by the following result:

\begin{theorem}
\label{characterizationcoquasi}
Let $H$ be a monoid and counital comagma such that conditions (b1), (b2) and (b3) of Definition \ref{Weak-Hopf-coquasigroup} hold. The following assertions are equivalent:
\begin{itemize}
\item[(i)] $H$ is a weak Hopf coquasigroup.
\item[(ii)] The morphisms $h=q_{R}^{2}\co \gamma\co j_{L}^{2}:H\times_{L}^{2} H\rightarrow H\times_{R}^{2} H$ and 
$s=q_{L}^{1}\co \beta\co j_{R}^{1}:H\times_{R}^{1} H\rightarrow H\times_{L}^{1} H$ are isomorphisms. Moreover, the morphism $j_{L}^{2}\co h^{-1}\co q_{R}^{2}$ is almost left $H$-colinear and $j_{R}^{1}\co s^{-1}\co q_{L}^{1}$ is almost right $H$-colinear.
\end{itemize}

\end{theorem}

When particularizing to Hopf coquasigroups, we get the assertion $(2)$ of the Theorem 2.5 given by Brzezinski in \cite{Brz}.
\begin{corollary}
\label{corolarioHcoqg}
Let $H$ be a monoid and counital comagma such that $\varepsilon_H$ and $\delta_H$ are morphisms of unital magmas (equivalently, $\eta_H$ and $\mu_H$ are morphisms of counital comagmas). Then $H$ is a Hopf coquasigroup if and only if the right and left Galois morphisms $\beta$ and $\gamma$ are isomorphisms and they have almost right $H$-colinear and almost left $H$-colinear inverses, respectively.
\end{corollary}

We will finish this paper giving the corresponding characterization for weak Hopf algebras.

\begin{corollary}
\label{corolarioHweakcoquasi}
Let $H$ be a monoid and comonoid such that conditions (b1), (b2) and (b3) of Definition \ref{Weak-Hopf-coquasigroup} hold. The following assertions are equivalent.
\begin{itemize}
\item[(i)] $H$ is a weak Hopf algebra.
\item[(ii)] The morphism $h$ defined in Theorem \ref{characterizationcoquasi} is an isomorphism.
\item[(iii)] The morphism $s$ defined in Theorem \ref{characterizationcoquasi} is an isomorphism.
\end{itemize}
 \end{corollary}

\section*{Acknowledgements}
The  authors were supported by  Ministerio de Econom\'{\i}a y Competitividad of Spain (European Feder support included). Grant MTM2013-43687-P: Homolog\'{\i}a, homotop\'{\i}a e invariantes categ\'oricos en grupos y \'algebras no asociativas.

\end{document}